\documentclass[11pt]{article}
\usepackage[english]{babel}
\usepackage{amssymb,amsmath,enumerate,a4,empheq}
\usepackage{latexsym}
\usepackage{amsthm}
\usepackage{graphicx}
\usepackage{epsfig}
\usepackage{pstricks}
\usepackage{pst-node}
\usepackage{colortab}
\usepackage{color}
\usepackage[nocompress,noadjust,space]{cite}

\newcommand{\vol}{\mathrm{vol}}
\newcommand{\bb}{\mathbb}
\newcommand{\R}{\bb R}
\newcommand{\Q}{\bb Q}
\newcommand{\Z}{\bb Z}
\newcommand{\N}{\bb N}
\newcommand{\floor}[1]{\left\lfloor#1\right\rfloor}
\newcommand{\ceil}[1]{\left\lceil#1\right\rceil}
\newcommand{\conv}{\operatorname{conv}}

\newcommand{\reccone}{\operatorname{rec}}

\newcommand{\affhull}{\operatorname{aff}}
\newcommand{\interior}{\operatorname{int}}

\newcommand{\aff}{\operatorname{aff}}
\newcommand{\sm}{\setminus}

\newcommand{\AFFIHULL}{\operatorname{\mathsf{AFF-IHULL}}}
\newcommand{\TSAT}{\operatorname{3\mathsf{SAT}}}
\newcommand{\NP}{\operatorname{\mathsf{NP}}}

\newcommand{\relint}{\operatorname{relint}}

\newtheorem{Prop}{Proposition}

\newtheorem{lemma}[Prop]{Lemma}

\newtheorem{cor}[Prop]{Corollary}
\newtheorem{thm}[Prop]{Theorem}

\newtheorem*{claim-no-number}{Claim}

\newenvironment{pf}{\begin{trivlist} \item[] {\em Proof.}}
{\hspace*{\stretch{1}} $\Box$ \end{trivlist}}

\newenvironment{cpf}{\begin{trivlist} \item[] {\em Proof.}}
{\hspace*{\stretch{1}} $\diamond$ \end{trivlist}}

\title{On the convergence of the affine hull\\ of the Chv\'atal-Gomory closures\thanks{This work was supported by the {\em Progetto di Eccellenza 2008--2009} of {\em Fondazione Cassa di Risparmio di Padova e Rovigo}.}}

\author{Gennadiy Averkov\thanks{Institut f\"ur Mathematische Optimierung, Fakult\"at f\"ur Mathematik, Otto-von-Guericke-Universit\"at Magdeburg, Germany.}
\and Michele Conforti\thanks{Dipartimento di Matematica, Universit\`a degli Studi di Padova, Italy.}
\and Alberto Del Pia\thanks{IFOR, Department of Mathematics, ETH Z\"urich, Switzerland.}
\and Marco Di Summa\footnotemark[3]
\and Yuri Faenza\thanks{Institut de math\'ematiques d'analyse et applications, EPFL, Lausanne, Switzerland.}}

\begin{document}

\maketitle

\begin{abstract}
Given an integral polyhedron $P\subseteq\R^n$ and a rational polyhedron $Q\subseteq\R^n$ containing the same integer points as $P$, we investigate how many iterations of the Chv\'atal-Gomory closure operator have to be performed on $Q$ to obtain a polyhedron contained in the affine hull of $P$. We show that if $P$ contains an integer point in its relative interior, then such a number of iterations can be bounded by a function depending only on $n$. On the other hand, we prove that if $P$ is not full-dimensional and does not contain any integer point in its relative interior, then no finite bound on the number of iterations exists.\\[4pt]
\end{abstract}

\section{Introduction}

The {\em integer hull} $Q_I$ of a polyhedron $Q \subseteq \R^n$ is
defined by $Q_I := \conv(Q \cap \Z^n)$, where ``$\conv{}$'' denotes the convex hull operator.
If $Q$ is a rational polyhedron, then $Q_I$ is a polyhedron \cite{Me} (see also \cite[\S16.2]{sch}).
A polyhedron $Q$ is \emph{integral} if $Q=Q_I$. Given an integral polyhedron $P\subseteq\R^n$, a \emph{relaxation} of $P$ is a rational polyhedron $Q \subseteq \R^n$ such that $Q\cap\Z^n=P\cap\Z^n$. An inequality $cx\le\floor{\delta}$ is a \emph{Chv\'atal-Gomory inequality} (\emph{CG inequality} for short) for a polyhedron $Q\subseteq \R^n$ if $c$ is an integer vector and $cx\le\delta$ is valid for $Q$. Note that $cx\le\floor\delta$ is a valid inequality for $Q\cap\Z^n$ and thus also for $Q_I$. The \emph{CG closure} $Q'$ of $Q$ is the set of points in $Q$ that satisfy all the CG inequalities for $Q$. If $Q$ is a rational polyhedron, then $Q'$ is again a rational polyhedron \cite[Theorem~1]{sch80}.
For $k\in\N$, the \emph{$k$--th CG closure} $Q^{(k)}$ of $Q$ is defined iteratively as $Q^{(k)}:=\left(Q^{(k-1)}\right)'$, with $Q^{(0)}:=Q$. If $Q$ is a rational polyhedron, then there exists a nonnegative integer $p$ such that $Q^{(p)}=Q_I$ \cite[Theorem~2]{sch80}. In other words, the sequence of polyhedra $\left(Q^{(k)}\right)_{k\in \N}$ finitely converges to $Q_I$. The minimum nonnegative integer $p$ for which $Q^{(p)}=Q_I$, called the \emph{CG rank} of $Q$ and denoted by $r(Q)$, can be viewed as the rate of finite convergence of the sequence $\left(Q^{(k)}\right)_{k\in \N}$ to $Q_I$.

%Given a set $S\subseteq\R^n$, the {\em affine hull} of $S$, denoted $\aff(S)$, is the smallest affine subspace containing $S$. The {\em dimension} of $S$ is the dimension of $\aff(S)$. The set $S$ is {\em full-dimensional} if its dimension is $n$.

It is well known that, already in dimension $2$, the CG rank of rational polyhedra can be arbitrarily high. In fact, for $t \in \N$, consider the polyhedron
$Q_t:=\conv(\{(0,0),(0,1),(t,1/2)\})$
(see Figure~\ref{fig:2dim}). Note that $(Q_t)_I=\conv\{(0,0),(0,1)\}$ for all $t \in \N$. It is folklore that
$Q_t\subseteq Q_{t+1}'$, hence by induction $r(Q_t)\geq t$. This example shows actually something stronger: it can take arbitrarily many rounds of the CG closure already for a polyhedron to be contained in the affine hull of its integer points (the {\em affine hull} of a set $S\subseteq\R^n$ is the smallest affine set containing $S$, and throughout the paper it will be denoted by $\aff(S)$). The purpose of this paper is to provide a systematic study of this phenomenon. More precisely, we consider the following question: given an integral polyhedron $P\subseteq\R^n$, is there an integer $p$ such that, for each relaxation $Q$ of $P$, one has $Q^{(p)}\subseteq\aff(P)$?
The example from Figure~\ref{fig:2dim} shows that
such a $p$ does not always exist. However, as our main result, we prove that if $P$ contains an integer point in its relative interior, then such a $p$ exists, and indeed it depends only on $n$.

\begin{figure}
\center
% Generated with LaTeXDraw 2.0.8
% Wed May 30 17:50:26 CEST 2012
% \usepackage[usenames,dvipsnames]{pstricks}
% \usepackage{epsfig}
% \usepackage{pst-grad} % For gradients
% \usepackage{pst-plot} % For axes
\scalebox{.8} % Change this value to rescale the drawing.
{
\begin{pspicture}(0,-2.12)(9.72,2.12)
\definecolor{color738b}{rgb}{0.8,0.8,0.8}
\definecolor{color739b}{rgb}{0.6,0.6,0.6}
\definecolor{color740b}{rgb}{0.4,0.4,0.4}
\psline[linewidth=0.04cm,linestyle=dashed,dash=0.16cm 0.16cm,arrowsize=0.05291667cm 2.0,arrowlength=1.4,arrowinset=0.4]{<-}(1.0,2.1)(1.0,-2.1)
\psline[linewidth=0.04cm,linestyle=dashed,dash=0.16cm 0.16cm,arrowsize=0.05291667cm 2.0,arrowlength=1.4,arrowinset=0.4]{->}(0.0,-0.9)(9.7,-0.88)
\pspolygon[linewidth=0.04,fillstyle=solid,opacity=0.5,fillcolor=color738b](1.0,1.16)(1.0,-0.84)(6.9,0.08)
\pspolygon[linewidth=0.04,fillstyle=solid,opacity=0.5,fillcolor=color739b](0.98,1.12)(0.98,-0.88)(4.96,0.1)
\pspolygon[linewidth=0.04,fillstyle=solid,fillcolor=color740b](0.98,1.14)(0.98,-0.86)(2.92,0.04)
\psdots[dotsize=0.4](1.0,-0.9)
\psdots[dotsize=0.4](1.0,1.1)
\psdots[dotsize=0.4,dotangle=-90.0](3.0,-0.9)
\psdots[dotsize=0.4,dotangle=-90.0](5.0,-0.9)
\psdots[dotsize=0.4,dotangle=-90.0](7.0,-0.9)
\psdots[dotsize=0.4,dotangle=-90.0](9.0,-0.9)
\psdots[dotsize=0.4,dotangle=-90.0](3.0,1.1)
\psdots[dotsize=0.4,dotangle=-90.0](5.0,1.1)
\psdots[dotsize=0.4,dotangle=-90.0](7.0,1.1)
\psdots[dotsize=0.4,dotangle=-90.0](9.0,1.1)
\end{pspicture}
}
\caption{In increasingly lighter shades of grey, polytopes $Q_1$, $Q_2$, and $Q_3$.}\label{fig:2dim}
\end{figure}
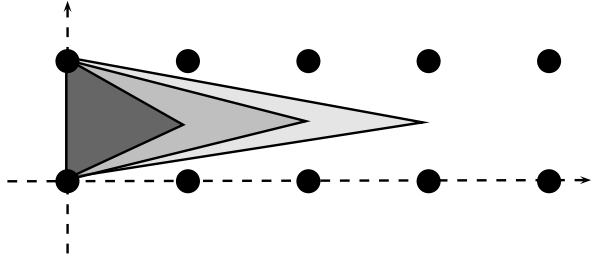

\begin{thm}\label{th:main}
There exists a function $\beta: \N \rightarrow \N$ such that, for each integral polyhedron $P \subseteq \R^n$ containing an integer point in its relative interior and each relaxation $Q$ of $P$, $Q^{(\beta(n))}$ is contained in $\aff(P)$.
\end{thm}

The above theorem can be interpreted as follows: for the family of rational polyhedra $Q\subseteq\R^n$ such that $Q_I$ contains an integer point in its relative interior, there is a global upper bound on the rate of finite convergence of the sequence $\left(\aff(Q^{(k)})\right)_{k\in\N}$ to $\aff(Q_I)$, where this global upper bound depends on $n$ only.
Theorem~\ref{th:main} is proved in Section~\ref{sec:thr1}, after some preliminaries in Section~\ref{sec:tools}.

We complement Theorem~\ref{th:main} by generalizing the construction from Figure~\ref{fig:2dim} to any
dimension. More precisely, we show in Section~\ref{sec:thr2} that if
$P\subseteq\R^n$ is not full-dimensional (i.e., its dimension is smaller than $n$) and does not contain any integer point in its relative interior, then there is no $p$ such that $Q^{(p)}\subseteq\aff(P)$ for every relaxation $Q$ of $P$.

\begin{thm}\label{th:main2}
If a non-full-dimensional integral polyhedron $P\subseteq\R^n$ does not contain any integer point in its relative interior, then for each $k \in \N$ there exists a relaxation $Q$ of $P$ such that $Q^{(k)}$ is not contained in $\aff(P)$.
\end{thm}

Let us remark that Theorems \ref{th:main} and \ref{th:main2} show a qualitative difference between those non-full-dimensional rational polyhedra that have an integer point in the relative interior of their integer hull, and those which do not. However, we prove that even for the former class, we cannot hope in general to converge to the affine hull after a small number of CG closures: in Section~\ref{sec:bounds} we show the best choice for the function $\beta$ from Theorem~\ref{th:main} is doubly exponential in $n$ ($\beta(n)=2^{2^{\Theta(n)}}$).

We conclude in Section~\ref{sec:complexity}, where we prove the $\NP$-completeness of a decision problem, denoted $\AFFIHULL$, which is closely related to the questions considered above: given $x \in \Q^n$ and a rational polyhedron $Q\subseteq\R^n$ (described by a system of linear inequalities with rational coefficients), decide whether $x \in \aff(Q_I)$.

To the best of our
knowledge, the questions we investigate in this paper have not been
addressed before. In fact, most works on the CG rank of polyhedra aim at
bounding the rank of rational polyhedra contained in the 0-1 cube (see, e.g., \cite{Eisc,PoSc,PoSt,RothSan}).

\section{Tools}\label{sec:tools}

In this section we provide some definitions and auxiliary results that will be used in the proofs of the main theorems.
For standard background material on convex sets, polyhedra, geometry of numbers and integer optimization, we refer the reader to the monographs \cite{Grun,sch,Bar,schn}.

Throughout the paper, $n$ will be a positive integer denoting the dimension of the ambient space.
Given a closed convex set $C\subseteq \R^n$, we denote by $\interior(C)$ the interior of $C$, by $\relint(C)$ the relative interior of $C$, and by $\reccone(C)$ the recession cone of $C$. We say that $C$ is \emph{lattice-free} if $\interior(C) \cap\Z^n= \varnothing$, and \emph{relatively lattice-free} if $\relint(C)\cap\Z^n=\varnothing$. Note that if $C$ is not lattice-free, then it is full-dimensional. A \emph{convex body} is a closed, convex, bounded subset of $\R^n$ with non-empty interior. A set $C\subseteq\R^n$ is {\em centrally symmetric} with respect to a given point $x\in C$ (or centered at $x$) when, for every $y\in\R^n$, one has $x+y\in C$ if and only if $x-y\in C$. When talking about distance and norm, we always refer to the Euclidean distance and the Euclidean norm; in particular, we denote the latter using the standard notation $\|\cdot\|$. For $d\in\{1,\dots,n\}$, when referring to the volume of a $d$--dimensional convex set $C\subseteq\R^n$, denoted $\vol(C)$, we shall always mean its $d$--dimensional volume, that is, the Lebesgue measure with respect to the affine subspace $\aff(C)$ of the
Euclidean space $\R^n$.
For $i\in\N$, we write $e^i$ to denote the unit vector of suitable dimension with 1 in its $i$-th entry and 0 elsewhere.
The all-zero vector of appropriate dimension is denoted by $\mathbf0$.

\subsection{Integer points in convex sets}

We will make use of Minkowski's Convex Body Theorem, which we state below (see, e.g., \cite[Chapter 7, \S3]{Bar}).

\begin{thm}[Minkowski's Convex Body Theorem]\label{thm:MCB}
Let $C\subseteq\R^n$ be a centrally symmetric convex body centered at the origin. If $\vol(C)\ge2^n$, then $C$ contains a non-zero integer point.
\end{thm}

The \emph{lattice width} of a closed convex set $C\subseteq\R^n$ is defined (for the integer lattice $\Z^n$) by
\[w(C) := \inf_{c \in \Z^n \sm \{\mathbf0\}}\left\{\sup_{x \in C} cx - \inf_{x \in C} cx \right\}.\]
If $C$ is full-dimensional and $w(C)<+\infty$, then there exists a non-zero integer vector $c$ for which
\[w(C) = \max_{x \in C} cx - \min_{x \in C} cx .\]
The following theorem is due to Khintchine~\cite{Khi} (see also \cite{balipasz,KanLov} and \cite[Chapter~7, \S8]{Bar} for related results and improvements).

\begin{thm}[Flatness Theorem]\label{thm:flat}
For every convex body $C \subseteq \R^n$ with $C \cap \Z^n=\varnothing$, one has $w(C)\leq \omega(n)$, where $\omega$ is a function depending on $n$ only.
\end{thm}

The following is a simple corollary.
\begin{cor}\label{cor:flat}
For every $k\in\N$ and every convex body $C \subseteq \R^n$ with $|C \cap k\Z^n|=1$, one has $w(C)\leq \omega(n,k)$, where $\omega$ is a function depending on $n$ and $k$ only.
\end{cor}

\begin{pf}
Wlog we can assume that the only point in $C \cap k\Z^n$ is the vector $(k,\dots,k)$.
This implies that $\frac{1}{2k}C$ contains no integer point.
By Theorem~\ref{thm:flat}, $w\left(\frac{1}{2k}C\right)$ is bounded by some function depending only on the dimension $n$.
As $w(C) = 2k \cdot w\left(\frac{1}{2k}C\right)$, $w(C)$ is bounded by some function depending only on $n$ and $k$.
\end{pf}

The above proof shows that one can set $\omega(n,k)=2 k \omega(n)$,
with $\omega(n)$ as in Theorem~\ref{thm:flat}. Since one can choose $\omega(n)=O\left(n^{3/2}\right)$ (see \cite[Proposition~2.3 and Theorem~2.4]{balipasz}), we obtain $\omega(n,k)=O\left(kn^{3/2}\right)$ (as both $n,k\to\infty$). We remark that a slightly worse bound of $O\left(kn^2\right)$ follows from \cite[Theorem~(4.1)]{KanLov}.

We will need the result of Corollary~\ref{cor:flat} for (possibly unbounded) rational polyhedra rather than convex bodies (which are bounded by definition). Thus we will make use of the following result.

\begin{cor}\label{cor:flat-poly}
For every $k\in\N$ and every rational polyhedron $Q \subseteq \R^n$ with $|Q \cap k\Z^n|=1$, one has $w(Q)\leq \omega(n,k)$, where $\omega$ is a function depending on $n$ and $k$ only.
\end{cor}

\begin{pf}
Since $|Q \cap k \Z^n | =1$, we have $\left|\frac{1}{k} Q \cap \Z^n\right| =1$.
So $\frac{1}{k} Q$ is a rational polyhedron containing precisely one
integer point, which we denote by $z$. If $\frac{1}{k} Q$ were unbounded,
then the sum of $z$ and any of the infinitely many integer vectors in the
recession cone of $\frac{1}{k} Q$ would
be in $\frac{1}{k} Q$, contradicting $\left|\frac{1}{k} Q \cap \Z^n\right| =1$.
Hence $\frac{1}{k} Q$ is bounded, which implies that $Q$ is bounded as well.
Then the assumptions of Corollary~\ref{cor:flat} are fulfilled for Q and thus $w(Q) \le \omega(n,k)$.
\end{pf}

Corollary~\ref{cor:flat-poly} can also be proven as a consequence of Corollary~\ref{cor:flat} and an observation of Eisenbrand and Shmonin \cite{EiSh}, who noticed that the Flatness Theorem remains true if $C$ is assumed to be a rational polyhedron instead of a convex body (see the discussion after Theorem 2.1 in \cite{EiSh}).

\subsection{Upper bounds on the CG rank}

The next lemma provides an important property of CG inequalities; for its proof see, e.g., \cite[page~340]{sch}.

\begin{lemma}\label{lem:Chv-of-face}
Let $Q\subseteq \R^n$ be a rational polyhedron and $F$ be a face of $Q$. Then for each $t \in \N$, $F^{(t)}=F \cap Q^{(t)}$.
\end{lemma}

Using the Flatness Theorem and the previous lemma, one can prove the following result \cite[Theorem $1'$]{CCT} (see also \cite[Theorem 23.3]{sch}).

\begin{lemma}\label{lem:PI=0}
The CG rank of every rational polyhedron $Q\subseteq \R^n$ with $Q\cap\Z^n=\varnothing$ is at most $\varphi(n)$, where $\varphi:\N\to\N$ is a function depending on $n$ only.
\end{lemma}

Wlog we assume that the values $\varphi(n)$ from Lemma~\ref{lem:PI=0} are non-decreasing in $n$.
We use the previous results to show the following upper bound on the CG rank.

\begin{lemma}\label{lem:upper}
For every rational polyhedron $Q \subseteq \R^n$
and every $c \in \Z^n$ and $\delta,\delta'\in\R$ (with $\delta'\ge\delta$) such that $cx \le
\delta$ is valid for $Q_I$ and $cx \le \delta'$ is valid for $Q$, the inequality $cx
\le \delta$ is valid for $Q^{(p+1)}$, where $p=(\lfloor \delta'\rfloor - \lfloor\delta
\rfloor) \theta(n)$ and $\theta:\N\to\N$ is a function depending on $n$ only.
\end{lemma}

\begin{pf}
We assume $n\ge2$, otherwise the statement is trivially verified.
We show that, for each integer $0\le k
\leq \lfloor\delta'\rfloor-\lfloor\delta\rfloor $, the inequality $cx\leq \lfloor \delta' \rfloor -k$ is valid for $Q^{(k\varphi(n-1)+k+1)}$, with
$\varphi$ being the function from
Lemma~\ref{lem:PI=0}.
Note that $Q^{(1)}\subseteq \{x \in \R^n: cx\leq
\lfloor \delta' \rfloor\}$. If $\lfloor\delta'\rfloor\le\delta$, the proof is complete.
Otherwise, the hyperplane $\{x \in \R^n: cx
= \lfloor \delta' \rfloor\}$ induces a face $F$ of $Q^{(1)}$
of dimension at most $n-1$ containing no integer points. From
Lemmas~\ref{lem:Chv-of-face} and~\ref{lem:PI=0} we conclude that $\{x
\in \R^n: cx = \lfloor \delta' \rfloor\}\cap
Q^{(\varphi(n-1)+1)}=\varnothing$, and consequently
$Q^{(\varphi(n-1)+2)}\subseteq \{x \in \R^n: cx\leq \lfloor \delta'
\rfloor -1\}$. By iterating this argument, we obtain
$Q^{(k\varphi(n-1)+k+1)}\subseteq \{x \in \R^n: cx\leq \lfloor \delta'
\rfloor - k\}$ for $0\le k \leq \lfloor \delta'\rfloor - \lfloor\delta\rfloor$.
In particular, when $k=\lfloor \delta'\rfloor - \lfloor\delta\rfloor$
one has that $cx \leq \lfloor\delta\rfloor$ is valid for $Q^{((\lfloor \delta'\rfloor - \lfloor\delta\rfloor)\theta(n)+1)}$
with $\theta(n):=\varphi(n-1)+1$.
\end{pf}

\subsection{Unimodular transformations}

A \emph{unimodular transformation} $u: \R^n \rightarrow \R^n$ maps a point $x \in \R^n$ to $u(x)=Ux + v$, where $U$ is a unimodular $n\times n$ matrix (i.e., a square integer matrix with $|\det(U)|=1$) and $v\in \Z^n$. It is well-known (see, e.g., \cite[Theorem~4.3]{sch}) that a non-singular matrix $U$ is unimodular if and only if so is $U^{-1}$. Furthermore, a unimodular transformation is a bijection of both $\R^n$ and $\Z^n$ that preserves $n$--dimensional volumes. Moreover, the following holds (see \cite[Lemma~4.3]{Eisc}).

\begin{lemma}\label{obs:unimodular-preserves}
If $Q \subseteq \R^n$ is a rational polyhedron,
$t\in\N$ and $u:\R^n\to\R^n$ is a unimodular transformation, then $u\left(Q^{(t)}\right)=u(Q)^{(t)}$. In particular, the CG ranks of Q
and u(Q) coincide.

\end{lemma}

Because of the previous lemma, when investigating the CG rank of a $d$--dimensional rational polyhedron $Q\subseteq \R^n$ with $Q\cap\Z^n \neq \varnothing$, we can apply a suitable unimodular transformation and assume that $Q$ is contained in the rational subspace $\R^d \times \{0\}^{n-d}=\{x \in \R^n: x_{d+1}=x_{d+2}=\dots=x_n=0\}$.

\subsection{Centrally symmetric convex bodies in integral polyhedra}

%For $r \in \Q_+$, $x \in \R^n$ and an affine subspace $H$ of $\R^n$ of dimension $d$, the \emph{$d$--ball} (\emph{of radius $r$ lying on $H$ and centered at $x$}) is the set of points lying on $H$ whose distance from $x$ is at most $r$. A \emph{$d$--ellipsoid} is the image of a $d$--ball under some invertible affine transformation, and its \emph{center} is the image of $x$. In particular, a unimodular transformation maps ellipsoids into ellipsoids of the same volume.

%Next, we state the \emph{John Ellipsoid Theorem}, in its version for centrally symmetric convex bodies (see, e.g., \cite{Bar}).

%\begin{thm}[John Ellipsoid Theorem for centrally symmetric bodies]\label{thr:john-ell}
%Let $K\subseteq \R^n$ be a centrally symmetric convex body. Then there exists a unique $n$--ellipsoid of maximum volume contained in $K$. This ellipsoid is centered at the origin and has volume at least $\vol(K)/\sqrt n$.
%\end{thm}

Given a convex body $C\subseteq\R^n$ and a point $x\in \interior(C)$, we define the \emph{coefficient of asymmetry of $C$ with respect to $x$} to be the value
\[\gamma(C,x):=\max_{z \in \R^n:\|z\|=1}\:\frac{\max\{\lambda:x+\lambda z\in C\}}{\max\{\lambda:x-\lambda z\in C\}}.\]
Note that $\gamma(C,x)\ge 1$, and $\gamma(C,x)=1$ if and only if $C$ is centrally symmetric with respect to $x$. Furthermore, assuming for the sake of simplicity that $x$ is the origin, $\gamma(C,x)$ is the smallest number $\gamma$ for which $\frac1\gamma C\subseteq -C$.

The following result was proved by Pikhurko~\cite[Theorem 4]{pik}.

\begin{thm}\label{thr:pik}
There exists a function $\sigma: \N \rightarrow ]0,+\infty[$ such that every non-lattice-free integral polytope $P\subseteq\R^n$ contains a point $x\in\interior(P)\cap\Z^n$ satisfying $\gamma(P,x)\le \sigma(n)$.
\end{thm}

%The following is probably known.

%\begin{Prop}
%Let $S,R\subseteq \R^n$ be finite sets, $x \in \R^n$. If $x \in \interior(P')$, then it can be expressed as the strictly convex combination of points of $S$ plus a conic combination of points of $R$. If conversely it can be expressed as the
%\end{Prop}
%\begin{proof}
%Let $S=\{s^1,\dots, s^{p}\}$, $R=\{r^1,\dots, r^{q}\}$. To show necessity, let $x \in \interior(P)$ be given by $\gamma_{1}s^1 + \dots + \gamma_{p}s^p + \mu_{1}r^1 + \dots + \mu_{q}r^q$, with $\gamma$ (resp. $\mu$) be the coefficient of a convex (resp. conic) combination. To show that all coefficients $\gamma$ are strictly positive it is enough to provide, for each $i=1,\dots,p$, a combination as above producing $\bar x$ and having $\gamma_i>0$. For $i=1,\dots, p$ either $\bar x\in s^i + \cone(R)$, or the line passing through $\bar x$ and $s^i$ intersects $P$ in some point $\bar z\neq s^i,\bar x$, and $\bar x$ is the proper convex combination of $s^i$ and $\bar z$. In both cases the thesis follows.
%\end{proof}

The main result of this section is the following.

\begin{thm}\label{lem:ellipsoid}
There exists a function $\nu: \N \rightarrow ]0,+\infty[$ such that every non-lattice-free integral polyhedron $P\subseteq\R^n$ contains a centrally symmetric convex body of volume $\nu(n)$, whose only integer point is its center.
\end{thm}

\begin{proof}
First of all we show that we can assume wlog that $P$ is bounded. This can be proven by using Steinitz' theorem (see, e.g., \cite[Theorem~1.3.10]{schn}), which is as follows:
given a set $T\subseteq\R^n$ and a point $z\in\interior(\conv(T))$, there exists a subset $T'\subseteq T$ with $|T'|\le 2n$ such that $z\in\interior(\conv(T'))$. Now, let $P$ be a non-lattice-free integral polyhedron and take $z\in\interior(P)\cap\Z^n$. Apply Steinitz' theorem with $T=P\cap\Z^n$ and note that $\conv(T')$ is a non-lattice-free integral polytope contained in $P$.
Therefore, replacing $P$ with $\conv(T')$, we may assume that $P$ is a polytope.

By Theorem~\ref{thr:pik}, there exists $x\in\interior(P)\cap\Z^n$ such that $\gamma(P,x)\le \sigma(n)$. To simplify notation, we write $\gamma$ instead of $\gamma(P,x)$ and assume wlog $x=\mathbf0$. Define $Q=\conv(P\cup-P)$. Note that $Q$ is a centrally symmetric full-dimensional integral polytope containing the origin in its interior.

We claim that there exists a centrally symmetric full-dimensional integral polytope $\bar Q\subseteq Q$ such that the origin is the unique integer point in $\interior(\bar Q)$. If the origin is the unique integer point in $\interior(Q)$, then we can take $\bar Q=Q$. Therefore we assume that there exists $y\in\interior(Q)\cap\Z^n$ with $y\ne\mathbf0$. Let $\{y,v^1,\dots,v^{n-1}\}$ be a basis of $\R^n$, where $v^1,\dots,v^{n-1}$ are vertices of $Q$ (the existence of such a basis follows from the full-dimensionality of $Q$). Then the polytope obtained as the convex hull of the points $\pm y,\pm v^1,\dots,\pm v^{n-1}$ is a centrally symmetric full-dimensional integral polytope contained in $Q$. Since $y$ does not lie in the interior of this polytope, we can iterate this procedure finitely many times until we obtain a centrally symmetric full-dimensional integral polytope $\bar Q\subseteq Q$ such that the origin is the unique integer point in $\interior(\bar Q)$.

If we define $S=\frac1\gamma \bar Q$, we have $S\subseteq\conv\left(\frac1\gamma P\cup-\frac1\gamma P\right)$. By the definition of the coefficient of asymmetry, $\pm\frac1\gamma P\subseteq P$, thus $S\subseteq P$. Furthermore,
\[\vol(S)=\frac1{\gamma^n}\vol(\bar Q)\ge\frac{2^n}{\gamma^n n!},\]
where the last inequality holds because $\bar Q$ is the union of $2^n$ full-dimensional integral simplices with disjoint interiors. Since $\gamma \le \sigma(n)$, we obtain the bound $\vol(S)\ge 2^n / (\sigma(n)^n n!)$, which depends only on $n$. Finally, the origin is the unique integer point in $S$, as $S\subseteq \bar Q$.
\end{proof}

Wlog we assume that the values $\nu(n)$ from Theorem~\ref{lem:ellipsoid} are non-increasing in $n$.

We remark that Theorem~\ref{lem:ellipsoid} does not hold if, instead of looking for {\em any} centrally symmetric convex body, we ask for the existence of a {\em specific} full-dimensional centrally symmetric convex body $S$ and a number $t>0$, both depending on $n$ only, such that $tS$ is contained in $P$ (up to a translation by an integer vector) and has its center as its unique integer point. This is shown by the following simple example. Let $(P_k)_{k \in \N}$ be the sequence of parallelograms in $\R^2$ given by $P_k:=\conv(\{\pm(k,1),\pm(1,0)\})$. The only integer point in $\interior(P_k)$ is the origin. One readily verifies that the distance between the origin and the boundary of the parallelogram cannot be lower-bounded by a constant. Then, if $S$ and $t$ are fixed as above, for $k$ large enough $tS$ is not contained in $P_k$.

%We remark that Theorem~\ref{lem:ellipsoid} does not hold if ``ellipsoid'' is replaced by ``ball'', as shown by the following simple example. Consider the family of polytopes $\{P_k\}_{k \in \N}\subseteq \R^2$, where $P_k$ is the triangle with vertices $(0,0)$, $(2,2k-1)$ and $(2,2k+1)$. The only integer point in $\interior(P_k)$ is $(1,k)$. We claim that the radius (and thus the volume) of a ball centered at $(1,k)$ and contained in $P_k$ cannot be lower bounded by a constant. To see this, observe that the volume of the triangle $P_k$ is 2 and the length of one of its sides is at least $2k$; this implies that the distance between any point in the triangle and this side is at most $2/k$.

\section{Proof of Theorem \ref{th:main}}\label{sec:thr1}

We show Theorem \ref{th:main} by induction on the codimension $n-d$ of $P$, the case $d=n$ being trivial. Hence we fix integers $d, n$ with $0\le d < n$.

Let $P\subseteq \R^n$ be a $d$--dimensional integral polyhedron that is not relatively-lattice-free. Up to unimodular transformations, we may assume that $P \subseteq \R^d \times \{0\}^{n-d}$ and $\mathbf0 \in \relint(P)$, since $P$ is a rational polyhedron with at least one integer point in its relative interior. Assuming for the moment $d>0$, there exists a $d$--dimensional centrally symmetric compact convex set $S$ contained in $P$, centered at some integer point of $P$, whose only integer point is its center and whose volume is $\nu(d)$ (see Theorem \ref{lem:ellipsoid}). Let $\pi : \R^n \rightarrow \R^{n-d}$ be the orthogonal projection onto the space of the last $n-d$ components. Note that $\pi(P)=\{\mathbf0\}$. Also, let $k=\lceil2^{n-1}n/\nu(n)\rceil$. Note that $k$ depends only on $n$.

\begin{claim-no-number}
$\pi(Q) \cap k\Z^{n-d} = \{\mathbf0\}$ for every relaxation $Q$ of $P$.
\end{claim-no-number}

\begin{cpf}
After translating $P$ by an integer vector, we may assume that $S$ is centered at $\mathbf0$. Assume by contradiction that the claim is false, i.e., there exist a relaxation $Q$ of $P$ and a point $\bar x \in \pi(Q) \cap k\Z^{n-d} \sm \{\mathbf0\}$ (see Figure \ref{fig:non-full-dim}).
Let $\hat x \in Q$ be such that $\pi(\hat x) = \bar x$. We can assume wlog that the integer vector $\bar x/k$ is a primitive vector, i.e, there are no integer points in the open segment $]\mathbf0,\bar x/k[$. Then, by applying a suitable linear unimodular transformation in the
space $\R^n$ which keeps the first $d$ components (and thus also $P$)
unchanged, we may assume that $\bar x/k=e^1\in\R^{n-d}$.

We define the $(d+1)$--dimensional compact convex set $C := \conv(S\cup\{\hat x\}) \subseteq \R^n$.
Note that $C \subseteq Q$, as both $S$ and $\hat x$ are contained in $Q$. Let $\bar C$ be the $(d+1)$--dimensional centrally symmetric compact convex set defined as $\bar C=C\cup-C$. Note that $\bar C$ lies in the space of the first $d+1$ components. The volume of $\bar C$ can be bounded as follows:
\[\vol(\bar C)\ge\frac{2k\nu(d)}{d+1}\ge\frac{2k\nu(n)}{n}\ge2^n,\]
where the last inequality follows from the choice of $k$.
By Minkowski's Convex Body Theorem (Theorem~\ref{thm:MCB}), $\bar C$ contains two non-zero integer points $z^1$ and $z^2$
with $z^1=-z^2$.
Since $S \cap \Z^n = \{\mathbf0\}$ and $\bar C \cap \affhull (P) = S$, $z^1$ and $z^2$ do not lie in $\affhull(P)$ and so they are not contained in $P$. By symmetry of $\bar C$, we can assume that $z^1$ lies in $C$.
As $C \subseteq Q$, $z^1$ is an integer point contained in $Q$ and not in $P$, contradicting the fact that $Q$ is a relaxation of $P$.
\end{cpf}

\begin{figure}
\centering
% Generated with LaTeXDraw 2.0.8
% Fri May 18 16:03:00 CEST 2012
% \usepackage[usenames,dvipsnames]{pstricks}
% \usepackage{epsfig}
% \usepackage{pst-grad} % For gradients
% \usepackage{pst-plot} % For axes
\scalebox{.8} % Change this value to rescale the drawing.
{
\begin{pspicture}(0,-3.8)(10.42,3.8)
\definecolor{color81b}{rgb}{0.8,0.8,0.8}
\definecolor{color82b}{rgb}{0.6,0.6,0.6}
\pspolygon[linewidth=0.04,fillstyle=solid,fillcolor=color81b](3.54,3.02)(3.54,-1.5)(1.16,-3.78)(1.16,0.64)
\psbezier[linewidth=0.04,fillstyle=solid,fillcolor=color82b](3.02,0.10832092)(3.02,-0.70335793)(1.74,-1.86)(1.74,-1.0483211)(1.74,-0.23664233)(3.02,0.9199998)(3.02,0.10832092)
\psline[linewidth=0.04cm,linestyle=dotted,dotsep=0.16cm,arrowsize=0.15cm 2.0,arrowlength=1.4,arrowinset=0.4]{->}(2.32,-0.6)(10.4,-0.56)
\psline[linewidth=0.04cm,linestyle=dotted,dotsep=0.16cm,arrowsize=0.15cm 2.0,arrowlength=1.4,arrowinset=0.4]{->}(2.34,-0.62)(2.34,3.78)
\psline[linewidth=0.04cm,linestyle=dotted,dotsep=0.16cm,arrowsize=0.15cm 2.0,arrowlength=1.4,arrowinset=0.4]{->}(2.36,-0.58)(0.0,-2.76)
\psdots[dotsize=0.24](2.36,-0.6)
\usefont{T1}{ptm}{m}{n}
\rput(2.5714064,-0.55){$0$}
\usefont{T1}{ptm}{m}{n}
\rput(2.0873437,-1){\large $S$}
\usefont{T1}{ptm}{m}{n}
\rput(1.5573437,-3.045){\large $P$}
\usefont{T1}{ptm}{m}{n}
\rput(7.8,-0.95){$\bar x$}
\usefont{T1}{ptm}{m}{n}
\rput(6.4414062,0.99){$\hat x$}
\psdots[dotsize=0.24](7.56,-0.58)
\psdots[dotsize=0.24](6.32,0.64)
\psline[linewidth=0.04cm](2.9,0.38)(6.26,0.64)
\psline[linewidth=0.04cm](2.14,-1.26)(6.28,0.62)
\psline[linewidth=0.04cm](1.96,-0.38)(6.28,0.64)
\usefont{T1}{ptm}{m}{n}
\rput(4.787344,0.875){\large $C$}
\psline[linewidth=0.04cm, linestyle=dashed,dash=0.08cm 0.08cm](2.92,-0.32)(6.32,0.66)
\psline[linewidth=0.04cm,linestyle=dotted,dotsep=0.16cm](7.56,-0.58)(6.4,-1.74)
\psline[linewidth=0.04cm,linestyle=dotted,dotsep=0.16cm](1.18,-1.76)(6.36,-1.78)
\psline[linewidth=0.04cm,linestyle=dotted,dotsep=0.16cm](6.36,-1.68)(6.34,0.7)
\psline[linewidth=0.04cm,linestyle=dashed,dash=0.16cm 0.16cm,arrowsize=0.15cm 2.0,arrowlength=1.4,arrowinset=0.4]{<-}(7.42,-0.48)(6.48,0.48)
\end{pspicture}
}
\caption{Illustration of the proof of the claim for the case $d=2$ and $n=3$.}\label{fig:non-full-dim}\end{figure}
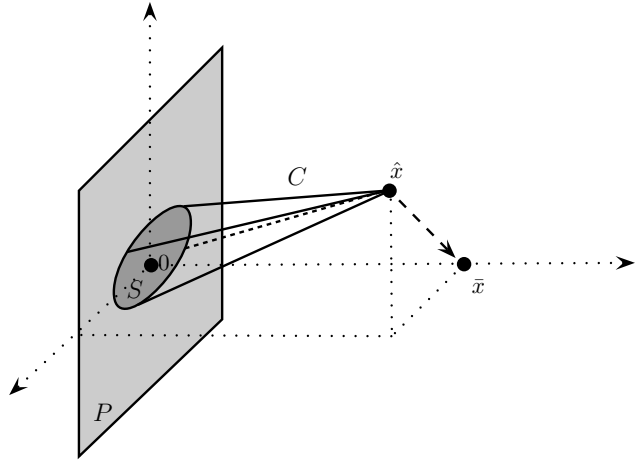

Note that though the above arguments cannot be used when $d=0$, the claim also holds for $d=0$, as in this case $P=\{\mathbf0\}$.

By Corollary~\ref{cor:flat-poly} and the above claim, for each relaxation $Q$ of $P$ there exists a number $\psi(n)$ depending only on $n$ such that $w(\pi(Q)) \le \psi(n)$. Now fix a relaxation $Q$ and let $c \in \Z^{n-d} \sm \{\mathbf0\}$ be such that
$$w(\pi(Q)) = \max_{x \in \pi(Q)} cx - \min_{x \in \pi(Q)} cx.$$
Define $c' = (0,\dots,0,c) \in \Z^n$.
Note that
$$\max_{x \in Q} c'x = \max_{x \in \pi(Q)} cx \qquad \text{and} \qquad \min_{x \in Q} c'x = \min_{x \in \pi(Q)} cx,$$
and $c'$ is orthogonal to $\affhull (P)$, thus $c'x=0$ for every $x \in P$.
It follows by Lemma \ref{lem:upper} that the equation $c'x=0$ is valid for $Q^{(p')}$, where $p' = \psi(n) \theta(n)+1$. Therefore, the value $p'$ depends on $n$ only.
By construction, $c'$ is a primitive vector. Hence, up to linear unimodular transformations, we may assume that $c' = e^n$. Thus $P$ can be written as $P = \bar P \times \{0\}$ for an integral polyhedron $\bar P \subseteq \R^{n-1}$.
Similarly, $Q^{(p')}$ can be written as $Q^{(p')} = \bar Q \times \{0\}$, for a rational polyhedron $\bar Q \subseteq \R^{n-1}$.
%Let $\bar P$ and $\bar Q$ be the orthogonal projection of $P$ and $Q^{(p')}$ respectively in the space $\R^{n-1}$ of the first $n-1$ variables.
%Note that $P = \bar P \times \{0\}$, and $Q^{(p')} = \bar Q \times \{0\}$.
Since the codimension $n-1-d$ of $\bar P \subseteq \R^{n-1}$ is smaller than the codimension $n-d$ of $P \subseteq \R^n$, it follows by induction that there exists $p \in \N$, which again only depends on $n$, such that $Q^{(p+p')} \subseteq \affhull(P)$. This concludes the proof of Theorem~\ref{th:main}.

\section{Proof of Theorem \ref{th:main2}}\label{sec:thr2}

A lemma of Chv\'atal, Cook, and Hartmann \cite[Lemma~2.1]{ChCoHa} gives sufficient conditions for a sequence of points to be in successive Chv\'atal closures of a rational polyhedron. The one we provide next is a less general, albeit sufficient for our needs, version of their original lemma.

\begin{lemma}\label{lem:iterat-cc}
Let $Q\subseteq \R^n$ be a rational polyhedron, $x \in Q$, $v \in \R^n$ and $p \in \N$. For $j \in \{1,\dots, p\}$, define $x^j:=x-j\cdot v$. Assume that, for all $j \in \{1,\dots,p\}$ and every inequality $cx \leq \delta$ valid for $Q_I$ with $c \in \Z^n$ and $cv < 1$, one has $cx^j\leq \delta$. Then $x^j \in Q^{(j)}$ for all $j \in \{1,\dots, p\}.$
\end{lemma}

We now prove Theorem \ref{th:main2}. Let $P\subseteq\R^n$ be a relatively lattice-free integral polyhedron $P$ of dimension $d<n$. Up to unimodular transformations, we can assume that $P \subseteq \R^{n-1} \times \{0\}$. Let $\bar x$ be a point in the relative interior of $P$. Fix $k \in \N$ and define $Q$ as the (topological) closure of $\conv(P\cup\{\bar x+(k+1)e^n\})$. Since $Q$ is a rational polyhedron (see \cite{balas}), $Q$ is a relaxation of $P$. We now argue that $Q^{(k)}$ is not contained in $\aff(P)$. We apply Lemma \ref{lem:iterat-cc} with $x=\bar x + (k+1) e^n$, $v=e^n$ and $p=k+1$. Let $cx \le \delta$ be an inequality
with $c \in \Z^n$. If $cx\le\delta$ is valid for $Q_I=P$ and satisfies $cv=c_n<1$, then $c_n\leq 0$, since both $c$ and $v$ are integer vectors. Then, for $1\le j\le k+1$, \[cx^j=\sum_{i=1}^{n-1} c_i x^j_i + (k+1-j)c_n = c\bar x + (k+1-j)c_n \leq \delta.\]
Lemma~\ref{lem:iterat-cc} gives in particular $x^k \in Q^{(k)}$. Since $x^k\notin\aff(P)$, this concludes the proof of Theorem~\ref{th:main2}.

\section{Asymptotic behavior of $\beta(n)$}
\label{sec:bounds}

In this section we prove that if in Theorem~\ref{th:main} one chooses the minimum $\beta(n)$ for $n\in\N$, then $\beta(n)=2^{2^{\Theta(n)}}$.

\begin{Prop}\label{prop:ub-beta}
For the function $\beta$ from Theorem~\ref{th:main}, one can set $\beta(n)=2^{4^{n+o(n)}}$.
\end{Prop}

\begin{pf}
The proof of Theorem~\ref{th:main} shows that one can set $\beta(n)= n(\psi(n)\theta(n)+1)$. Here $\theta$ is the function from Lemma~\ref{lem:upper}, while $\psi(n)=\omega(k,n)$, with $k=\ceil{2^{n-1}n/\nu(n)}$, where $\omega$ and $\nu$ are the functions from Corollary~\ref{cor:flat-poly} and Theorem~\ref{lem:ellipsoid}, respectively.

Let us first give an upper bound on $\theta(n)$. The proof of Lemma~\ref{lem:upper} shows that we can set $\theta(n)=\varphi(n-1)+1$, where $\varphi$ is the function from Lemma~\ref{lem:PI=0}. Since one can choose $\varphi(n)= n^{3n}$ (see \cite[Remark~(2)]{CCT}), we can set $\theta(n)=(n-1)^{3(n-1)}+1=O\left(n^{3n}\right)$.

We now consider the function $\psi(n)=\omega(k,n)$. As pointed out in Section~\ref{sec:tools}, one can set $\omega(k,n)=O\left(kn^{3/2}\right)$. We need an upper bound on $k=\ceil{2^{n-1}n/\nu(n)}$, i.e., a lower bound on $\nu(n)$. The proof of Theorem~\ref{lem:ellipsoid} shows that one can choose $\nu(n)=\frac{2^n}{\sigma(n)^n n!}$, where $\sigma$ is the function from Theorem~\ref{thr:pik}. As shown in \cite[Theorem~4]{pik}, one can set $\sigma(n)=8n\cdot 15^{2^{2n+1}}$. Therefore we can choose $\psi(n)=O\left(15^{n\cdot2^{2n+1}}8^nn^{n+5/2}n!\right)$.

Overall, we can set $\beta(n)= n(\psi(n)\theta(n)+1)=O\left(15^{n\cdot2^{2n+1}}8^nn^{4n+7/2}n!\right)=2^{4^{n+o(n)}}$.
\end{pf}

We now prove that any lower bound on $\beta(n)$ is doubly exponential in $n$.

\begin{Prop}
Every function $\beta$ from Theorem~\ref{th:main} satisfies $\beta(n) \ge 2^{2^{n-2}}-2$ for all $n\in\N$.
\end{Prop}

\begin{pf}
We make use of a construction from \cite[Remark 3.10]{avwawe}.
The so-called \emph{Sylvester sequence} $(t_n)_{n \in\N}$ is defined by
\[
	t_n:=
	\begin{cases}
		2 & \text{if} \ n=1, \\
		(t_{n-1} - 1) t_{n-1} + 1 & \text{otherwise.}
	\end{cases}
\]
In the following we assume $n\ge2$, as otherwise the statement is trivial.

One easily checks by induction that
\[
	\frac{1}{t_1} + \cdots + \frac{1}{t_{n-1}} + \frac{1}{t_n-1} =1.
\]
It follows that the simplex
\[
	T:=\conv (\{\mathbf0,t_1 e^1, \ldots, t_{n-1} e^{n-1}, (t_n-1) e^n \}) \subseteq \R^n
\]
is lattice-free (actually, even maximal lattice-free; see \cite{avwawe}). We modify the simplex $T$ using an appropriate perturbation of its vertex $(t_n-1) e^n$. E.g., we can consider the simplex
\[
	Q:= \conv (\{\mathbf0,t_1 e^1,\ldots,t_{n-1} e^{n-1}, p_\varepsilon \}),
\]
 where $p_\varepsilon$ is a rational point in $\interior(T)$, $\|p_\varepsilon - (t_n-1) e^n \| \le \varepsilon$ and $\varepsilon>0$. By construction, the integer hull of $Q$ is the facet
\[
	P:=\conv (\{\mathbf0,t_1 e^1,\ldots,t_{n-1} e^{n-1} \}) \subseteq \R^n
\]
of $T$. Clearly, $P$ is an integral polytope of dimension $n-1$ with the integer point $(1,\ldots,1,0) \in \R^n$ in its relative interior. For a small $\varepsilon$  (say,  $\varepsilon<1$) the width of $Q$ in direction $e^n$ is greater than $t_n-2$. Applying Lemma~\ref{lem:iterat-cc} in the same manner as in the proof of Theorem~\ref{th:main2}, we obtain $Q^{(i)} \ne P$ for $i=t_n-2$. Since, as shown in \cite[page 1026]{Lazi}, $t_n \ge 2^{2^{n-2}}$ for every $n\in\N$, we have $\beta(n) \ge t_n -2  \ge 2^{2^{n-2}}-2$.
\end{pf}

As one can see from our proofs, the doubly exponential behavior of the minimal $\beta(n)$ is inherited from the doubly exponential behavior of $\sigma(n)$ from Theorem~\ref{thr:pik}. On the other hand, the choice of an upper bound for $\omega(n)$ and $\varphi(n)$ is not that relevant: we arrive at the assertion of Proposition~\ref{prop:ub-beta} as long as $\omega(n)$ and $\varphi(n)$ are chosen to be of order $2^{2^{o(n)}}$ (that is,
very weak bounds for $\omega(n)$ and $\varphi(n)$ would suffice).

\section{$\NP$-completeness of $\AFFIHULL$}
\label{sec:complexity}

Let $\AFFIHULL$ denote the following decision problem: given $x \in \Q^n$ and a rational polyhedron $Q\subseteq\R^n$ (described by a system of linear inequalities with rational coefficients), decide whether $x \in \aff(Q_I)$.

\begin{Prop}\label{prop:NPcomplete}
	$\AFFIHULL$ is $\NP$-complete.
\end{Prop}
\begin{proof}
We first show that $\AFFIHULL \in \NP$. Let $(x,Q)$ be an arbitrary instance of problem $\AFFIHULL$. Let $d$ be the dimension of $Q$. Using \cite[Theorem~17.1]{sch} one can show the existence of vectors $z^0,\ldots,z^d\in\Q^n$ whose description size is polynomial in the description size of $Q$ and such that $\aff (Q_I) = \aff (\{z^0,\ldots,z^d\})$. It is known that there exists a polynomial algorithm which verifies, for given $x, z^0,\ldots,z^d\in\Q^n$, whether $x \in \aff(\{z^0,\ldots,z^d\})$ (in fact, the latter condition is reduced to solving a system of linear equalities). Thus the sequence $z^0,\ldots,z^d$ can be taken as the certificate, and we get $\AFFIHULL \in \NP$.

We conclude the proof by showing that $\TSAT$ is polynomially reducible to $\AFFIHULL$. Let $\phi$ be an arbitrary 3CNF formula with $n$ variables $x_1,\ldots,x_n$ and $m$ clauses. It is well-known that satisfiability of $\phi$ can be expressed by a system of linear inequalities. We interpret $x_1,\ldots,x_n$ as 0/1-variables. Each clause from $\phi$ generates a linear inequality. E.g., if $\phi$ contains the clause $x_1 \wedge x_2 \wedge \overline{x_3}$, then we introduce the inequality $x_1 + x_2 + (1-x_3) \ge 1$ (we proceed similarly for other possible clauses). In this way we construct inequalities $a^j x \ge b_j$ with $j \in \{1,\dots,m\}$, where $a^j \in \Z^n$ and $b_j \in \Z$. From the fact that each clause uses at most three variables we get:
	\begin{align}
		a^j x & \in \{-3,\ldots,3\} & & \forall x \in \{0,1\}^n \ \forall j \in \{1,\dots,m\}, \label{aj:x} \\
		b_j & \in \{-2,\ldots,1\} & & \forall j \in \{1,\dots,m\}. \label{bj}
	\end{align}
	We introduce an additional variable $x_{n+1}$ and define the following system for variables $x_1,\ldots,x_{n+1}$:
	\begin{empheq}{align}
			 a^j x & \ge b_j - 4 (1-x_{n+1})  & &\forall j \in \{1,\dots,m\}, \label{cond:1} \\
			0 & \le x_i \le 3 - 2 x_{n+1} & &\forall i \in \{1,\dots,n\}, \label{cond:2} \\
			0 & \le x_{n+1}, & & \label{cond:3}\\
			x_i & \in \Z & &\forall i \in \{1,\dots,n+1\}. \label{cond:4}
	\end{empheq}
Note that in \eqref{cond:1} $x$ denotes the vector
$(x_1,\ldots,x_n)$ (that is, the variable $x_{n+1}$ is not included as a
component).
	For the above system the following conditions can be verified in a straightforward way.
	\begin{enumerate}[(a)]
	 	\item Every solution of the system satisfies $x_{n+1} \in \{0,1\}$ (see \eqref{cond:2}, \eqref{cond:3} and \eqref{cond:4}).
		\item Each element of $\{0,1\}^n \times \{0\}$ is a solution of the system (in view of \eqref{aj:x} and \eqref{bj}).
		\item Every solution with $x_{n+1}=1$ lies in $\{0,1\}^n \times \{1\}$ (see \eqref{cond:2} and \eqref{cond:4}).
		\item The system has a solution lying in $\{0,1\}^n \times \{1\}$ if and only if the 3CNF formula $\phi$ is satisfiable (see \eqref{cond:1}).
	\end{enumerate}
	Let $Q$ be the rational polyhedron in $\R^{n+1}$ defined by \eqref{cond:1}, \eqref{cond:2} and \eqref{cond:3}. From (a), (b) and (c) we see that $\aff(Q_I)$ is either $\R^n \times \{0\}$ or $\R^{n+1}$. Thus, using (d), $\phi$ is satisfiable if and only if $e^{n+1} \in \aff(Q_I)$. This shows that $\TSAT$ is polynomially reducible to $\AFFIHULL$.
\end{proof}

Since the absolute values of the coefficients in the
system \eqref{cond:1}--\eqref{cond:4} are at most 4,
we have actually shown that $\AFFIHULL$ is $\NP$-complete even in the strong sense (i.e.,
also in the case where the integer values used to define instances of
$\AFFIHULL$ are represented in the unary encoding; see, e.g., \cite[\S4.2]{GarJohn}).

\end{document}